\documentclass{article}
\pdfoutput=1
\newif\iflong
\longtrue 

\title{On directed homotopy equivalences and a notion of 
directed topological complexity}
\author{Eric Goubault\footnote{LIX, Ecole Polytechnique, CNRS, Universit\'e Paris-Saclay, 91128 Palaiseau, France goubault@lix.polytechnique.fr. Work presented during
the Hausdorff
Institute ``Applied Computational Algebraic Topology'' semester, on the 14th
September 2017 for which the author gratefully acknowledges the organizers.}}
\date{\today}

\usepackage{url} 
\usepackage{a4wide}
\usepackage{amsmath,amssymb,amsthm}
\usepackage{tikz}
\usetikzlibrary{matrix}
\def\Z{{\Bbb Z}}
\def\R{{\Bbb R}}

\newcommand\inn[1]{\mathfrak{I}(#1)}

\def\one{\mbox{$1 \!\! 1$}}
\newcommand\PP[0]{\mathcal{P}}
\newcommand\pth[1]{P#1}
\newcommand\diTC[1]{\overrightarrow{TC}(#1)}

\newcommand\map[3]{#1 : #2 \rightarrow #3}
\newtheorem{lemma}{Lemma}
\newtheorem{theorem}{Theorem}
\newtheorem{proposition}{Proposition}

\newtheorem{example}{Example}
\newtheorem{definition}{Definition}

\newcommand\comment[1]{}
\newcommand\Ab{\mathbf{Ab}}

\newcommand\ForAuthors[1]
 {\par\smallskip                     
  \begin{center}
   \fbox
   {\parbox{0.9\linewidth}
    {\raggedright--- #1}
   }
  \end{center}
  \par\smallskip                     %
 }        

\begin{document}

\maketitle

\begin{abstract}
This short note introduces a notion of directed homotopy equivalence (or
dihomotopy equivalence)
and of 
``directed'' topological complexity (which elaborates on the notion that can be 
found in e.g. \cite{farber}) which have  
a number of desirable joint properties. 
In particular, being dihomotopically equivalent 
implies having bisimilar natural homologies (defined
in 
\cite{naturalhomology}). Also, under mild conditions, directed topological 
complexity is an invariant of our directed homotopy equivalence and having 
a directed topological complexity equal to one is (under these conditions) equivalent
to being dihomotopy equivalent to a point (i.e., to being ``dicontractible'', as in the
undirected case). 
It still remains to compare this notion with the notion introduced in 
\cite{CSL16}, which has lots of good properties as well. 
For now, it seems that for reasonable spaces, this new proposal of
directed homotopy equivalence identifies more spaces than the one
of \cite{CSL16}. 
\end{abstract}

\tableofcontents

\newpage 

\section{Introduction}

The aim of this note is to introduce another notion of directed homotopy 
equivalence than the one of \cite{CSL16}, hoping to get other insights on 
directed topological spaces. The view we are taking here is that of topological
complexity, as defined in \cite{farber}, adapted to directed topological spaces. 

Let us briefly motivate the interest of this ``directed'' topological complexity notion. 
In the very nice work of M. Farber, it is observed that the very important
planification problem in robotics boils down to, mathematically speaking, 
finding a section to the path space fibration $\chi \ : \ PX=X^I \rightarrow X \times X$ with
$\chi(p)=(p(0),p(1))$. If this section is continuous, then the complexity is the lowest
possible 
(equal to one), otherwise, the minimal number of discontinuities that would encode
such a section would be what is called the topological complexity of $X$. This
topological complexity is both understandable algorithmically, and topologically,
e.g. as $s$ having a continuous section is equivalent to $X$ being contractible. 
More generally speaking, the topological complexity is defined as the Schwartz
genus of the path space fibration, i.e. is the minimal cardinal of partitions of 
$X \times X$ into ``nice'' subspaces $F_i$ such that $s_{F_i} \ : \ F_i \rightarrow
PX$ is continuous. 

This definition perfectly fits the planification problem in robotics where there are 
no constraints on the actual control that can be applied to the physical apparatus
that is supposed to be moved from point $a$ to point $b$. In many applications, 
a physical apparatus may have dynamics that can be described as an ordinary
differential equation in the state variables $x \in \R^n$ and in time $t$, 
parameterized by control parameters $u \in \R^p$, $\dot{x}(t)=f(t,x(t))$. These parameters are generally
bounded within some set $U$, and, not knowing the precise control law (i.e. 
parameters $u$ as a function of time $t$) to be applied, the way the controlled system
can evolve is as one of the solutions of the differential inclusion 
$\dot{x}(t) \in F(t,x(t))$ where $F(t,x(t))$ is the set of all $f(t,x(t),u)$ with
$u \in U$. Under some classical conditions, this differential inclusion can be
proven to have solutions on at least a small interval of time, but we will not discuss
this further here. Under the same conditions, the set of solutions of this differential
inclusion naturally generates a dspace (a very general structure of directed
space, where a preferred subset of paths is singled out, called
directed paths, see e.g. \cite{grandisbook}). 
Now, the planification problem in the presence of control constraints equates to finding
sections to the analogues to the path space fibration\footnote{That would most probably
not qualify for being called a fibration in the directed setting.} taking
a dipath to its end points. This is developed in next section, and we introduce
a notion of directed homotopy equivalence that has precisely, and in a certain
non technical sense, minimally, the right properties with respect to this directed
version of topological complexity. 

The development of the notion of directed topological complexity and of its properties,
together with applications to optimal control is joint work with coauthors, and
will be published separately.

\paragraph{Mathematical context : }

The context is that of d-spaces \cite{grandisbook}. 

\begin{definition}[\cite{grandisbook}]
A directed topological space, or d-space $X=(X,dX)$ is a topological space equipped with
a set $dX$ of continuous maps $p:I \rightarrow X$ (where $I=[0,1]$ is the unit
segment with the usual topology inherited from $\R$), called directed paths or
d-paths, satisfying three axioms~: 
\begin{itemize}
\item every constant map $I\rightarrow X$ is directed
\item $dX$ is closed under composition with non-decreasing maps from $I$ to $I$
\item $dX$ is closed under concatenation
\end{itemize}
\end{definition}

For $X$ a d-space, let us note by $PX$ (resp. $TX$) the topological space, with compact
open topology, of dipaths (resp. the
trace space, ie. $PX$ modulo increasing homeomorphisms of the unit directed interval) in $X$. 
$PX(a,b)$ (resp. $TX(a,b)$) is the sub-space of $PX$ (resp. of $TX$)
containing only dipaths (resp. traces) from point $a \in X$
to point $b \in X$. We write * for the concatenation map from $PX(a,b)\times PX(b,c)$
to $PX(a,c)$ (resp. on trace spaces), which is continuous. 

A dmap $f$ from d-space $X$ to d-space $Y$ is a continuous map from $X$ to $Y$ that also maps elements from $dX$ to elements
of $dY$ (i.e. they preserve directed paths). 

In what follows, we will be particularly concerned with the following map~: 

\begin{definition}
We define the {\em dipath space map}\footnote{By analogy with the classical path space
fibration - but fibration may be a bad term in that case in directed algebraic topology.}
$\chi : PX \rightarrow X \times X$ of $X$ by  
$\chi(p)=(p(0),p(1))$ for $p \in PX$.
\end{definition}

Because $PX$ only contains directed paths, the image of $\chi$ is just a subset of $X \times X$, 
called $\Gamma_X=\{ (x,y) \ | \ \exists p \in PX, \ p(0)=x, \ p(1)=y \ \}$. 
On the classical case, we do not need to force the restriction to the image
of the path space fibration, since the notions of contractibility and path-connectedness
are simple enough to be defined separately. In the directed setting, dicontractibility,
and ``directed connectedness'' are not simple notions and will be defined 
here through the study of the dipath space map. 

In order to study this map, in particular when looking at conditions under which there exists ``nice'' sections to it, we need a few concepts
from directed topology. 



\section{Some useful directed topological constructs}

Let $X$ be a d-space. 
\comment{We can associate with $X$ the natural system $\tilde{X} : FT(X) \rightarrow Top_*$
($Top_*$ being the category of pointed topological spaces) 
with~: 
\begin{itemize}
\item $\tilde{X}(u)=(TX(u_0,u_1),u)$, where $u_0$ is the initial point of $u$ (i.e. $u(0)$), 
and $u_1$ its final point (i.e. $u(1)$)
\item $\tilde{X}(\alpha,\beta)(a)=\alpha * a * \beta$, where $(\alpha,\beta)$ is a morphism
from $u$ to $v=\alpha * u * \beta$ and $a$ is a trace from $u_0$ to $u_1$ (i.e. an element
of $TX$). 
\end{itemize}
}
We define as in \cite{Eilenberg}, $\preceq$ the preorder on $X$, $x \preceq y$ iff there exists a dipath from $x$ to $y$. We define the category $\PP{X}$ whose:
	\begin{itemize}
		\item objects are pairs of points $(x,y)$ of $X$ such that $x \preceq y$ (i.e. objects
are elements of $\Gamma_X$)
		\item morphisms (called extensions) from $(x,y)$ to $(x',y')$ are pairs $(\alpha, \beta)$ of dipaths of $X$ with $\alpha$ going from $x'$ to $x$ and $\beta$ going from $y$ to $y'$
	\end{itemize}

We now define, for each $X$ d-space, the functor $\pth{X}$ from $\PP{X}$ to $Top$ with~:
\begin{itemize}
\item $\pth{X}(x,y)=PX(x,y)$
\item $\pth{X}(\alpha,\beta)(a)=\alpha * a * \beta$, where $(\alpha,\beta)$ is a morphism
from $(x,y)$ to $(x',y')$
and $a$ is a trace from $x$ to $y$ (i.e. and element
of $TX(x,y)$). 
\end{itemize}

\paragraph{Remark : } There is an obvious link to profunctors and to enriched
category theory that we will not be contemplating here. Similar ideas from
enriched category theory 
in directed algebraic topology have already been used in \cite{porter08,porter15} and \cite{CSL16}. 

\vskip .2cm

A homotopy 
is a continuous function $H : I\times X \rightarrow Y$. 
We say that two maps $f,g : X \rightarrow Y$ are {homotopic} if there is a homotopy $H$ such that 
$H(0,\_) = f$ and $H(1,\_) = g$. This is an equivalence relation, compatible with composition. 

A {d-homotopy equivalence} is a \textit{dmap} $\map{f}{X}{Y}$ which is invertible up to 
homotopy, i.e., such that there is a {dmap} $\map{g}{Y}{X}$ with $f\circ g$ and $g\circ f$ homotopic to identities. We say that two {d}spaces are {d-homotopy equivalent} if there is a {d-}homotopy equivalence between them.

In such a case, $f$ and $g$ being dmaps induce 
$$Pf \ : \ PX \rightarrow PY$$ resp. 
$$Pg \ : \
PY \rightarrow PX$$ 
which are
continuously bigraded maps in the sense that $Pf_{a,b} \ : \ PX(a,b) \rightarrow PY(f(a),f(b))$
and this bigrading is continuous in $a$, $b$
(resp. $Pg_{c,d} \ : \ PY(c,d) \rightarrow PX(g(c),g(d))$, continuous in $c$, $d$).

We write $PY^f$ for the sub topological space of $PY$ of dipaths from $f(a)$ to $f(b)$ in $Y$ 
for some $(a,b) \in \Gamma_X$ (resp. $PX^g$ for the sub topological space of $PX$
of dipaths in $X$ from $g(c)$ to $g(d)$ for some $(c,d) \in \Gamma_Y$).



\section{Dihomotopy equivalences}

\subsection{Dihomotopy equivalence  and dicontractibility}

\begin{definition}
\label{def:dihomotopyequiv}
Let $X$ and $Y$ be two d-spaces. A dihomotopy equivalence between $X$ and $Y$ is 
given by~: 
\begin{itemize}
\item A d-homotopy equivalence between $X$ and $Y$, $f: X \rightarrow Y$ and
$g : Y \rightarrow X$. 
\item A map $F \ : \ PY^f \rightarrow PX$ 
continuously bigraded as 
$F_{a,b} \ : \ PY(f(a),f(b)) \rightarrow PX(a,b)$
such that $(Pf_{a,b},F_{a,b})$ is a homotopy equivalence between $PX(a,b)$ and
$PY(f(a),f(b))$ 
\item 
A map $G \ : \ PX^g \rightarrow PY$, 
continuously bigraded as 
$G_{c,d} \ : \ PX(g(c),g(d)) \rightarrow PY(c,d)$
such that $(Pg_{c,d},G_{c,d})$ is a homotopy equivalence between
$PY(c,d)$ and $PX(g(c),g(d))$
\item These homotopy equivalences are natural in the following sense 
\begin{itemize}
\item For the two diagrams below (separately), for all $(\alpha,\beta) \in {\cal P}X$, there exists $(\gamma,\delta) \in 
  {\cal P}Y$ (with domains and codomains induced by the diagrams below) 
such that they commute\footnote{Meaning both squares, one
with $Pf_{a,b}$ and $Pf_{a',b'}$ and the other with $F_{a,b}$ and $F_{a',b'}$ ; and
similarly for the diagram on the right hand side.} up 
to homotopy
\begin{center}
\begin{minipage}{.45\linewidth}
  \begin{tikzpicture}[scale=1]
\matrix (m) [matrix of math nodes,row sep=3em,column sep=4em,minimum width=2em]
  {
     \scriptstyle PX(a,b) & \scriptstyle PY(f(a),f(b)) \\
     \scriptstyle PX(a',b') & \scriptstyle PY(f(a'),f(b')) \\};
  \path[-stealth]
    (m-1-1) edge node [left] {$\scriptstyle PX(\alpha,\beta)$} (m-2-1)
            edge [out=10,in=170] node [above] {${\scriptstyle Pf_{a,b}}$} (m-1-2)
(m-1-2) edge [out=-170,in=-10] node [below] {$\scriptstyle F_{a,b}$} (m-1-1)
(m-2-1) edge [out=10,in=170] node [above] {${\scriptstyle Pf_{a',b'}}$} (m-2-2)
(m-2-2) edge [out=-170,in=-10] node [below] {$\scriptstyle F_{a',b'}$} (m-2-1)
    (m-1-2) edge node [dashed,right] {$\scriptstyle PY(\gamma,\delta)$} (m-2-2);
\end{tikzpicture}
\end{minipage}
\begin{minipage}{.45\linewidth}
  \hfill
  \begin{tikzpicture}[scale=1]
\matrix (m) [matrix of math nodes,row sep=3em,column sep=4em,minimum width=2em]
  {
     \scriptstyle PX(g(c),g(d)) & \scriptstyle PY(c,d) \\
     \scriptstyle PX^g(u',v') & \scriptstyle PY(c',d') \\};
  \path[-stealth]
    (m-1-1) edge node [left] {$\scriptstyle PX(\alpha,\beta)$} (m-2-1)
            edge [out=10,in=170] node [above] {${\scriptstyle G_{c,d}}$} (m-1-2)
(m-1-2) edge [out=-170,in=-10] node [below] {$\scriptstyle Pg_{c,d}$} (m-1-1)
(m-2-1) edge [out=10,in=170] node [above] {${\scriptstyle G_{c',d'}}$} (m-2-2)
(m-2-2) edge [out=-170,in=-10] node [below] {$\scriptstyle Pg_{c',d'}$} (m-2-1)
    (m-1-2) edge node [dashed,right] {$\scriptstyle PY(\gamma,\delta)$} (m-2-2);
\end{tikzpicture}
\end{minipage}
\end{center}
with $g(c')=u'$ and $g(d')=v'$.
\item For the two diagrams below (separately), for all $(\gamma,\delta) \in {\cal P}Y$ there exists
$(\alpha,\beta) \in {\cal P}X$ 
such that they commute up 
to homotopy 
\begin{center}
\begin{minipage}{.45\linewidth}
  \begin{tikzpicture}[scale=1]
\matrix (m) [matrix of math nodes,row sep=3em,column sep=4em,minimum width=2em]
  {
     \scriptstyle PX(a,b) & \scriptstyle PY(f(a),f(b)) \\
     \scriptstyle PX(a',b') & \scriptstyle PY^f(u',v') \\};
  \path[-stealth]
    (m-1-1) edge node [left] {$\scriptstyle PX(\alpha,\beta)$} (m-2-1)
            edge [out=10,in=170] node [above] {${\scriptstyle Pf_{a,b}}$} (m-1-2)
(m-1-2) edge [out=-170,in=-10] node [below] {$\scriptstyle F_{a,b}$} (m-1-1)
(m-2-1) edge [out=10,in=170] node [above] {${\scriptstyle Pf_{a',b'}}$} (m-2-2)
(m-2-2) edge [out=-170,in=-10] node [below] {$\scriptstyle F_{a',b'}$} (m-2-1)
    (m-1-2) edge node [dashed,right] {$\scriptstyle PX(\gamma,\delta)$} (m-2-2);
\end{tikzpicture}
\end{minipage}
\begin{minipage}{.45\linewidth}
  \hfill
  \begin{tikzpicture}[scale=1]
\matrix (m) [matrix of math nodes,row sep=3em,column sep=4em,minimum width=2em]
  {
     \scriptstyle PX(g(c),g(d)) & \scriptstyle PY(c,d) \\
     \scriptstyle PX(g(c'),g(d')) & \scriptstyle PY(c',d') \\};
  \path[-stealth]
    (m-1-1) edge node [left] {$\scriptstyle PX(\alpha,\beta)$} (m-2-1)
            edge [out=10,in=170] node [above] {${\scriptstyle G_{c,d}}$} (m-1-2)
(m-1-2) edge [out=-170,in=-10] node [below] {$\scriptstyle Pg_{c,d}$} (m-1-1)
(m-2-1) edge [out=10,in=170] node [above] {${\scriptstyle G_{c',d'}}$} (m-2-2)
(m-2-2) edge [out=-170,in=-10] node [below] {$\scriptstyle Pg_{c',d'}$} (m-2-1)
    (m-1-2) edge node [dashed,right] {$\scriptstyle PX(\gamma,\delta)$} (m-2-2);
\end{tikzpicture}
\end{minipage}
\end{center}
with $f(a')=u'$ and $f(b')=v'$.
\end{itemize}
\end{itemize}
\end{definition}

We sometimes write $(f,g,F,G)$ for the full data associated to the dihomotopy
equivalence $f \ : \ X \rightarrow Y$. 
Note that in the definition above, we always have the following diagrams that commute on the nose, so that the conditions above only consists of 6 commutative diagrams
up to homotopy~:   

\begin{center}
\begin{minipage}{.3\linewidth}
  \hfill
  \begin{tikzpicture}[scale=1]
    \node (Fx') at (1,0) {$\scriptstyle \pth{X}(a',b')$};
    \node (Fx) at (1,1.5) {$\scriptstyle \pth{X}(a,b)$};
    \node (Gy') at (3,0) {$\scriptstyle \pth{Y}(f(a'),f(b'))$};
    \node (Gy) at (3,1.5) {$\scriptstyle \pth{Y}(f(a),f(b))$};
    \draw[->] (Fx) -- (Fx');
    \draw[->] (Fx) -- node [above] {$\scriptscriptstyle Pf_{a,b}$} (Gy);
    \draw[->] (Fx') -- node [below] {$\scriptscriptstyle Pf_{a',b'}$} (Gy');
    \draw[->] (Gy) -- (Gy');
    \node (Fi) at (0.7,0.75) {$\scriptscriptstyle \pth{X}(\alpha,\beta)$};
    \node (Gj) at (3.3,0.75) {$\scriptscriptstyle \pth{Y}(f(\alpha),f(\beta))$};
  \end{tikzpicture}
\end{minipage}
\begin{minipage}{.3\linewidth}
  \hfill
\begin{tikzpicture}[scale=1]
    \node (Fx') at (1,0) {$\scriptstyle \pth{X}(g(c'),g(d'))$};
    \node (Fx) at (1,1.5) {$\scriptstyle \pth{X}(g(c),g(d))$};
    \node (Gy') at (3,0) {$\scriptstyle \pth{Y}(c',d')$};
    \node (Gy) at (3,1.5) {$\scriptstyle \pth{Y}(c,d)$};
    \draw[->] (Fx) -- (Fx');
    \draw[<-] (Fx) -- node [above] {$\scriptscriptstyle Pg_{c,d}$} (Gy);
    \draw[<-] (Fx') -- node [below] {$\scriptscriptstyle Pg_{c',d'}'$} (Gy');
    \draw[->] (Gy) -- (Gy');
    \node (Fi) at (0.7,0.75) {$\scriptscriptstyle PX(g(\gamma),g(\delta))$};
    \node (Gj) at (3.3,0.75) {$\scriptscriptstyle PY(\gamma,\delta)$};
  \end{tikzpicture}
\end{minipage}
\end{center}



\paragraph{Remark : }
This definition clearly bears a lot of similarities with Dwyer-Kan weak equivalences
in simplicial categories (see e.g. \cite{bergner04}). The main ingredient
of Dwyer-Kan weak equivalences being exactly that $Pf$ induces a homotopy 
equivalence. But our definition adds continuity and 
``extension'' or ``bisimulation-like'' 
conditions to it, which are instrumental to our theorems and to the classification
of the underlying directed geometry.

\vskip .2cm

\paragraph{Remark : } 
There is an obvious notion of deformation diretract of $X \subseteq Y$, 
which is a dihomotopy
equivalence  
$r \ : \ Y \rightarrow X$ such that the inclusion map from $X$ to $Y$ is the 
left homotopy inverse of $r$ and $r \circ i=Id$. 
But there seems to be no reason that in general, two dihomotopically equivalent spaces 
are diretracts of a third one (the mapping cylinder in the classical
case). It may be true with a zigzag of diretracts though, as in \cite{CSL16}. 

\vskip .2cm

A 
dicontractible directed space is a space for which there exists a directed deformation retract
onto one of its points. Particularizing once again the definition above, we get~: 

\begin{definition}
\label{def:dicontractible}
Let $X$ be a d-space.
$X$ is dicontractible if  
there is a continuous map $R \ : \{*\} \rightarrow PX$, 
continuously bigraded, 
such that $R_{c,d}$ are homotopy equivalences (hence in particular, all path
spaces of $X$ are contractible).
\end{definition}


\vskip .2cm

\subsection{Strong dihomotopy equivalence}

Algebraically speaking, there is a simple condition that enforces the bisimulation
condition above (Lemma \ref{strongimpliesdihomequiv}), that we call strong dihomotopy equivalence (see below Definition
\ref{def:strongdihomequiv}). We do not know if it is equivalent to dihomotopy
equivalence for a large class of directed spaces.

\begin{definition}
\label{def:strongdihomequiv}
Let $X$ and $Y$ be two d-spaces. A strong dihomotopy equivalence between $X$ and $Y$ is given by~: 
\begin{itemize}
\item A d-homotopy equivalence between $X$ and $Y$, $f: X \rightarrow Y$ and
$g : Y \rightarrow X$. 
\item A map $F \ : \ PY^f \rightarrow PX$ 
continuously bigraded as 
$F_{a,b} \ : \ PY(f(a),f(b)) \rightarrow PX(a,b)$
such that $(Pf_{a,b},F_{a,b})$ is a homotopy equivalence between $PX(a,b)$ and
$PY(f(a),f(b))$ 
\item 
A map $G \ : \ PX^g \rightarrow PY$, 
continuously bigraded as 
$G_{c,d} \ : \ PX(g(c),g(d)) \rightarrow PY(c,d)$
such that $(Pg_{c,d},G_{c,d})$ is a homotopy equivalence between
$PY(c,d)$ and $PX(g(c),g(d))$ 
\item (a) For all $\alpha \in PX(a',a)$, $v \in PY(f(a),f(b))$ and $\beta \in PX(b,b')$, 
$F_{a',b'}(Pf_{a',a}(\alpha)*v*Pf_{b,b'}(\beta)) \sim \alpha * F_{a,b}(v) * \beta$
\item (b) 
For all $\gamma \in PY(c',c)$, $v \in PX(g(c),g(d))$ and 
$\delta \in PY(d,d')$, 
$G_{c',d'}(Pg_{c',c}(\gamma)*v*Pg_{d,d'}(\delta)) \sim \gamma * G_{g(c),g(d)}(v) 
* \delta$
\item (c) For all 
$\gamma \in PY^f(u',f(a))$, $p \in PY(f(a),f(b))$ and
$\delta \in PY^f(f(b),v')$, 
there exists $(a',b') \in {\cal P}X$ such that $f(a')=u'$, $f(b')=v'$ and 
$F_{a',b'}(\gamma * p * \delta) \sim F_{a',a}(\gamma) * F_{a,b}(p) * F_{b,b'}(\delta)$
\item (d) For all 
$\alpha \in PX^g(u',g(c))$, $q \in PX(g(c),g(d))$ and
$\beta \in PX^g(g(d),v')$, 
there exists $(c',d') \in {\cal P}Y$ such that $g(c')=u'$, $g(d')=v'$ and 
$G_{c',d'}(\alpha * q * \beta) \sim G_{c',c}(\gamma) * G_{c,d}(q) * G_{d,d'}(\delta)$
\end{itemize}
\end{definition}

\begin{lemma}
\label{strongimpliesdihomequiv}
Strong dihomotopy equivalences are dihomotopy equivalences
\end{lemma}

\begin{proof}
Let $f \ : \ X \rightarrow Y$ be a strong dihomotopy equivalence ; it comes
with $g \ : \ Y \rightarrow X$, $G \ : \ PX \rightarrow PY$ and $F \ : \ PY 
\rightarrow PX$. Let $(\alpha,\beta) \in {\cal P}X$, with $\alpha \in PX(a',a)$
and $\beta \in PX(b,b')$. 
In order to prove that $f$ is a dihomotopy equivalence, we must find $(\gamma,
\delta)$ such that the diagram below involving $F_{a,b}$ and $F_{a',b'}$ commutes
up to homotopy (the other diagram is commutative, for free) 
\begin{center}
  \begin{tikzpicture}[scale=1]
\matrix (m) [matrix of math nodes,row sep=3em,column sep=4em,minimum width=2em]
  {
     \scriptstyle PX(a,b) & \scriptstyle PY(f(a),f(b)) \\
     \scriptstyle PX(a',b') & \scriptstyle PY(f(a'),f(b')) \\};
  \path[-stealth]
    (m-1-1) edge node [left] {$\scriptstyle PX(\alpha,\beta)$} (m-2-1)
            edge [out=10,in=170] node [above] {${\scriptstyle Pf_{a,b}}$} (m-1-2)
(m-1-2) edge [out=-170,in=-10] node [below] {$\scriptstyle F_{a,b}$} (m-1-1)
(m-2-1) edge [out=10,in=170] node [above] {${\scriptstyle Pf_{a',b'}}$} (m-2-2)
(m-2-2) edge [out=-170,in=-10] node [below] {$\scriptstyle F_{a',b'}$} (m-2-1)
    (m-1-2) edge node [dashed,right] {$\scriptstyle PX(\gamma,\delta)$} (m-2-2);
\end{tikzpicture}
\end{center}
This commutes indeed with $\gamma=Pf_{a',a}(\alpha)$ and $\delta=Pf_{b,b'}(\beta)$
because of property (a) of strong dihomotopy equivalence $f$. 

Now, consider the diagram below, involving $G_{c,d}$ and $G_{c',d'}$ 

\begin{center}
  \begin{tikzpicture}[scale=1]
\matrix (m) [matrix of math nodes,row sep=3em,column sep=4em,minimum width=2em]
  {
     \scriptstyle PX(g(c),g(d)) & \scriptstyle PY(c,d) \\
     \scriptstyle PX^g(u',v') & \scriptstyle PY(c',d') \\};
  \path[-stealth]
    (m-1-1) edge node [left] {$\scriptstyle PX(\alpha,\beta)$} (m-2-1)
            edge [out=10,in=170] node [above] {${\scriptstyle G_{c,d}}$} (m-1-2)
(m-1-2) edge [out=-170,in=-10] node [below] {$\scriptstyle Pg_{c,d}$} (m-1-1)
(m-2-1) edge [out=10,in=170] node [above] {${\scriptstyle G_{c',d'}}$} (m-2-2)
(m-2-2) edge [out=-170,in=-10] node [below] {$\scriptstyle Pg_{c',d'}$} (m-2-1)
    (m-1-2) edge node [dashed,right] {$\scriptstyle PX(\gamma,\delta)$} (m-2-2);
\end{tikzpicture}
\end{center}
with $g(c')=u'$ and $g(d')=v'$. 
Let $\gamma=G_{c',d'}(\alpha)$ and $\delta=G_{d,d'}(\beta)$. Property $(d)$ implies
that the diagram above commutes up to homotopy.  


Similarly, let $(\gamma,\delta) \in {\cal P}Y$, property (c) of $f$ implies
that 
the following diagram 
commutes up to homotopy by taking $\alpha=F_{a',a}(\gamma)$ and $\beta=F_{b,b'}(\delta)$

\begin{center}
  \begin{tikzpicture}[scale=1]
\matrix (m) [matrix of math nodes,row sep=3em,column sep=4em,minimum width=2em]
  {
     \scriptstyle PX(a,b) & \scriptstyle PY(f(a),f(b)) \\
     \scriptstyle PX(a',b') & \scriptstyle PY^f(u',v') \\};
  \path[-stealth]
    (m-1-1) edge node [left] {$\scriptstyle PX(\alpha,\beta)$} (m-2-1)
            edge [out=10,in=170] node [above] {${\scriptstyle Pf_{a,b}}$} (m-1-2)
(m-1-2) edge [out=-170,in=-10] node [below] {$\scriptstyle F_{a,b}$} (m-1-1)
(m-2-1) edge [out=10,in=170] node [above] {${\scriptstyle Pf_{a',b'}}$} (m-2-2)
(m-2-2) edge [out=-170,in=-10] node [below] {$\scriptstyle F_{a',b'}$} (m-2-1)
    (m-1-2) edge node [dashed,right] {$\scriptstyle PX(\gamma,\delta)$} (m-2-2);
\end{tikzpicture}
\end{center}
with $f(a')=u'$ and $f(b')=v'$.

And finally, property (b) implies that the following diagram commutes up to
homotopy, by taking $\alpha=Pg_{c',d'}(\gamma)$ and $\beta=Pg_{d,d'}(\delta)$

\begin{center}
  \begin{tikzpicture}[scale=1]
\matrix (m) [matrix of math nodes,row sep=3em,column sep=4em,minimum width=2em]
  {
     \scriptstyle PX(g(c),g(d)) & \scriptstyle PY(c,d) \\
     \scriptstyle PX(g(c'),g(d')) & \scriptstyle PY(c',d') \\};
  \path[-stealth]
    (m-1-1) edge node [left] {$\scriptstyle PX(\alpha,\beta)$} (m-2-1)
            edge [out=10,in=170] node [above] {${\scriptstyle G_{c,d}}$} (m-1-2)
(m-1-2) edge [out=-170,in=-10] node [below] {$\scriptstyle Pg_{c,d}$} (m-1-1)
(m-2-1) edge [out=10,in=170] node [above] {${\scriptstyle G_{c',d'}}$} (m-2-2)
(m-2-2) edge [out=-170,in=-10] node [below] {$\scriptstyle Pg_{c',d'}$} (m-2-1)
    (m-1-2) edge node [dashed,right] {$\scriptstyle PX(\gamma,\delta)$} (m-2-2);
\end{tikzpicture}
\end{center}
\end{proof}

\subsection{Simple properties and examples of directed homotopy equivalences}

The first obvious (but important) observation is that directed homotopy equivalence
refines ordinary homotopy equivalence. Also, 
directed homotopy equivalence is an invariant of dihomeomorphic dspaces~: 

\begin{lemma}
Let $X$, $Y$ be two directed spaces. Suppose there exists $f \ : \ X \rightarrow Y$ a dmap,
which has an inverse, also a dmap. Then $X$ and $Y$ are directed homotopy equivalent.
\end{lemma}

\begin{proof}
Take $g=f^{-1}$, $F=Pg$ and $G=Pf$. This data forms a directed homotopy equivalence.
\end{proof}

Now, natural homology \cite{naturalhomology} is going to be an invariant of
dihomotopy equivalence, as it should be~: 

\begin{lemma}
\label{lemma:bisim}
Let $X$, $Y$ be two directed spaces. Suppose $X$ and $Y$ are directed homotopy equivalent.
Then $X$ and $Y$ have bisimilar natural homotopy and homology (in the sense of \cite{CSL16}). 
\end{lemma}

\begin{proof}
Call $f \ : \ X \rightarrow Y$ and $g \ : \ Y \rightarrow X$ the underlying dmaps, forming the
homotopy equivalence which is a directed homotopy equivalence. 

The bisimulation relation we are looking for is the relation~: 
$$\{ ((a,b),Pf_{a,b},(f(a),f(b))) \ | \ (a,b) \in \Gamma_X \} \cup \{ ((g(c),g(d)),Pg_{c,d},(c,d)) \ | \ (c,d) \in \Gamma_Y \}
$$
The diagrams defining the directed homotopy equivalence imply that $R$ is hereditary with respect
to extension maps. 
\end{proof}



Unfortunately, unlike Dwyer-Kan equivalences, or classical homotopy equivalences,
our dihomotopy equivalences do not have the 2-out-of-3 property. We only
have preservation by composition, as shown in next Lemma, but also, for surjective
dihomotopy equivalences, two thirds of the 2-out-of-3 property, as shown in 
Proposition \ref{bit2}. 

\begin{lemma}
Compositions of dihomotopy equivalences are dihomotopy equivalences. 
\end{lemma}

\begin{proof}
Suppose $f_1 \ : \ X \rightarrow Y$ and $f_2 \ : \ Y \rightarrow Z$ 
are dihomotopy equivalences.
We have quadruples 
$(f_1,g_1,F_1,G_1)$, 
$(f_2,g_2,F_2,G_2)$ as in Definition \ref{def:dihomotopyequiv}. 
Now, it is obvious to see that
its composite $(f_2 \circ f_1, g_1 \circ g_2,F_1 \circ F_2,G_2 \circ G_1)$ is a dihomotopy
equivalence from $X$ to $Z$. 
\end{proof}

The problem for getting a general 2-out-of-3 property on dihomotopy equivalences
can be exemplified as follows. 
Suppose that $f^1$ is a dihomotopy equivalence and that $f^2 \circ f^1$ is a
dihomotopy equivalence. 
In particular, by 2-out-of-3 on classical homotopy equivalence, we know that 
$f^1$ is a homotopy equivalence, with homotopy inverse $g^1$. 
Consider now the following composites, for all $(a,b)\in {\cal P}X$

\begin{center}
  \begin{tikzpicture}[scale=1]
\matrix (m) [matrix of math nodes,row sep=3em,column sep=4em,minimum width=2em]
  {
     \scriptstyle PX(a,b) & \mbox{$\scriptstyle PY(f^1(a),f^1(b))$} & 
\mbox{$\scriptstyle PZ(f^2 \circ f^1(a),f^2\circ f^1(b))$} \\
};
  \path[-stealth]
    (m-1-1) edge node [above] {\mbox{${\scriptstyle {Pf^1}_{a,b}}$}} (m-1-2)
(m-1-2) edge node [above] {\mbox{$\scriptstyle Pf^2_{f^1(a),f^1(b)}$}} (m-1-3)
    (m-1-1) edge [out=-10,in=-170] node [below] {\mbox{$\scriptstyle Pf^2 \circ f^1_{a,b}$}} (m-1-3);
\end{tikzpicture}
\end{center}
Because of 2-out-of-3 for classical homotopy equivalences, and as
$Pf^2\circ f^1$ and $Pf^1$ in the diagram above are homotopy equivalences, 
$Pf^2_{f^1(a),f^1(b)}$ is a homotopy equivalence. But we need it to be a homotopy
equivalence for all $(c,d) \in {\cal P}Y$, not only the ones in the image of $f^1$.
Similarly to Dwyer-Kan equivalences (see e.g. \cite{bergner04}), 
it is reasonable to add in the definition of
dihomotopy equivalences the assumption that $f$ is surjective on points. Then we
have

\begin{proposition}
\label{bit2}
If $f^1 \ : \ X \rightarrow Y$ 
and $f^2 \ : \ Y \rightarrow Z$ are such that
$f^1$ and $f^2 \circ f^1$ are surjective dihomotopy equivalences, 
then $f^2$ is a surjective dihomotopy equivalence. 
\end{proposition}

\begin{proof}
In that
case, we get that $Pf^2$ induces homotopy equivalences from all $PY(c,d)$ (call
the homotopy inverse $F^2$). 
Now, notice that, denoting by $F^{21}_{a,b}$ 
the homotopy inverse of $Pf^2\circ f^1_{a,b}=Pf^2\circ Pf^1_{a,b}$, and by $F^1_{a,b}$ the homotopy
inverse of $Pf^1_{a,b}$, $F^{2r} = Pf^1_{a,b} \circ F^{21}_{a,b}$ is a right homotopy 
inverse to $Pf^2_{f^1(a),f^1(b)}$, because
$$\begin{array}{rcl}
Pf^2_{f^1(a),f^1(b)} \circ (Pf^1_{a,b} \circ F^{21}_{a,b}) & = & (Pf^2_{f^1(a),f^1(b)} \circ Pf^1_{a,b})\circ F^{21}_{a,b}\\
& = & (P(f^2 \circ f^1)_{a,b}) \circ F^{21}_{a,b} \\
& \sim & Id_{PZ_{f^2 \circ f^1(a),f^2 \circ f^1(b)}} \\
\end{array}$$
We therefore have a right homotopy inverse $F^{2r}_{c,d}$ and a homotopy inverse
$F^{2}_{c,d}$ of $Pf^2_{c,d}$. So, 
$$\begin{array}{rcl}
F^{2r}_{c,d} & \sim & (F^2_{c,d} \circ Pf^2_{c,d}) \circ F^{2r}_{c,d} \\
 & = & F^2_{c,d} \circ (Pf^2_{c,d} \circ F^{2r}_{c,d}) \\
& \sim & F^2_{c,d}
\end{array}$$
and
$$\begin{array}{rcl}
F^{2r}_{c,d} \circ Pf^2_{c,d} & \sim & F^2_{c,d} \circ Pf^2_{c,d} \\
 & \sim & Id_{PY(c,d)}
\end{array}$$
therefore $F^{2r}_{c,d}$ is a homotopy inverse of $Pf^2_{c,d}$. Moreover, as
a composition of maps $Pf^1_{c,d}$ with $F^{21}_{c,d}$ which are continuous in 
$(c,d) \in {\cal P}Y$, it is a continuously bigraded map on $(c,d)$. 

We also have to look at the map induced by $g^2$. In the following diagram  
\begin{center}
  \begin{tikzpicture}[scale=1]
\matrix (m) [matrix of math nodes,row sep=3em,column sep=4em,minimum width=2em]
  {
     \mbox{$\scriptstyle PX(g^1 \circ g^2(e),g^1 \circ g^2(f))$} & 
\mbox{$\scriptstyle PY(g^2(e),g^2(f))$} & 
\mbox{$\scriptstyle PZ(e,f)$} \\
};
  \path[-stealth]
    (m-1-2) edge node [above] {\mbox{${\scriptstyle {Pg^1}_{g^2(e),g^2(f)}}$}} (m-1-1)
(m-1-3) edge node [above] {\mbox{$\scriptstyle Pg^2_{e,f}$}} (m-1-2)
    (m-1-3) edge [out=-170,in=-10] node [below] {\mbox{$\scriptstyle Pg^1 \circ g^2_{e,f}$}} (m-1-1);
\end{tikzpicture}
\end{center}
by 2-out-of-3 for classical homotopy equivalences, without any other assumptions,
we get that $Pg^2_{e,f}$ is a homotopy equivalence (with homotopy inverse 
$G^2_{e,f}$) for all path spaces $PZ(e,f)$. 
We note as above that, denoting by $G^{12}_{e,f}$ the homotopy inverse
of $Pg^1\circ g^2_{e,f}$, $G^{2l}_{e,f}=G^{12}_{e,f} \circ Pg^1_{g^2(e),g^2(f)}$ is
a left homotopy inverse of $Pg^2_{e,f}$. As before, we easily get that
$G^{2l}_{e,f}\sim G^2_{e,f}$ and $G^{2l}_{e,f}$ is a homotopy inverse of $Pg^2_{e,f}$
which forms a continuously bigraded map. 


Now we have to check the extension diagrams of Definition \ref{def:dihomotopyequiv}. 
Let $(\alpha,\beta) \in {\cal P}Y$ with $\alpha \in PY(c',c)$, $\beta \in 
PY(d,d')$. As $f^1$ is surjective, $c=f(a)$ and $d=f(b)$ for some $(a,b) \in
{\cal P}X$. As $f^1$ is a dihomotopy equivalence, we have the existence
of $(\alpha',\beta') \in {\cal P}X$ as in the diagram below. 
 
\comment{
\begin{center}
  \begin{tikzpicture}[scale=1]
\matrix (m) [matrix of math nodes,row sep=3em,column sep=4em,minimum width=2em]
  {
     \scriptstyle PX(a,b) & \scriptstyle PY(c,d) \\
     \scriptstyle PX(a',b') & \scriptstyle PY(c',d') \\};
  \path[-stealth]
    (m-1-1) edge node [left] {$\scriptstyle PX(\alpha',\beta')$} (m-2-1)
            edge node [above] {${\scriptstyle Pf_{a,b}}$} (m-1-2)
(m-2-1) edge node [above] {${\scriptstyle Pf_{a',b'}}$} (m-2-2)
    (m-1-2) edge node [dashed,right] {$\scriptstyle PY(\gamma,\delta)$} (m-2-2);
\end{tikzpicture}
\end{center}
}

Now, we use the fact that $f^2 \circ f^1$ is a dihomotopy equivalence, and
we get a map 
$(\alpha,\beta) \in 
  {\cal P}Z$ 
such that the following diagram commutes up to homotopy

\begin{center}
  \begin{tikzpicture}[scale=1]
\matrix (m) [matrix of math nodes,row sep=3em,column sep=4em,minimum width=2em]
  {
     \scriptstyle PX(a,b) & \scriptstyle PY(c,d) & \mbox{$\scriptstyle PZ(f^2(c),f^2(d))$} \\
     \scriptstyle PX(a',b') & \scriptstyle PY(c',d') & \mbox{$\scriptstyle PZ(f^2(c'),f^2 (d'))$} \\};
  \path[-stealth]
    (m-1-1) edge node [left] {$\scriptstyle PX(\alpha',\beta')$} (m-2-1)
            edge node [above] {${\scriptstyle Pf^1_{a,b}}$} (m-1-2)
(m-2-1) edge node [above] {${\scriptstyle Pf^1_{a',b'}}$} (m-2-2)
    (m-1-2) edge node [right] {$\scriptstyle PY(\gamma,\delta)$} (m-2-2)
    (m-1-3) edge node [right] {$\scriptstyle PZ(\alpha,\beta)$} (m-2-3)
(m-1-3) edge [in=30,out=150] node [above] {${\scriptstyle F^{21}_{c,d}}$} (m-1-1)
(m-2-3) edge [in=-30,out=-150] node [above] {${\scriptstyle F^{21}_{c,d}}$} (m-2-1)
(m-1-3) edge[dashed] node [above] {$\scriptstyle F^2_{c,d}$} (m-1-2)
(m-2-3) edge[dashed] node [above] {$\scriptstyle F^2_{c',d'}$} (m-2-2);
\end{tikzpicture}
\end{center}

In the diagram above, $F^2_{c,d}$ (represented as a dashed arrow) is actually
the composite $Pf^1_{a,b} \circ F^{21}_{c,d}$ as shown before ; similarly,
$F^2_{c',d'}$ is the composite $Pf^1_{a',b'} \circ F^{21}_{c',d'}$. Therefore we
have the extension property needed for $F^2_{c,d}$. The five other diagrams can
be proven in a similar manner, by pulling back or pushing forward the existence
of maps using the diagrams for $Pf^1$, $Pg^1$, $F^1$, $G^1$ (resp.
$Pf^2 \circ f^1$, $Pg^1 \circ g^2$, $F^{21}$ and $G^{12}$). 

\end{proof}

\paragraph{Remark : }
Similarly, if we have a surjective $f^2 \ : \ Y \rightarrow Z$ that 
is a dihomotopy equivalence, 
and $f^1 \ : \ X \rightarrow Y$ such that $f^2 \circ f^1$ is a surjective
dihomotopy
equivalence, first, there is no reason why $f^1$ should be surjective. 
Similarly as with $g^2$ before, we can prove that $Pf^1_{a,b}$
forms a continuous bigraded family of homotopy equivalences with a continuous
bigraded family of homotopy inverses. The problem is with $Pg^1$ which we
can only prove to have a continuous family of homotopy inverses on 
spaces $PY(g^2(e),g^2(f))$. The only result we can have in general is dual
to the one of Proposition \ref{bit2}. 
If $f^2 \ : \ X \rightarrow Y$ 
and $f^2 \ : \ Y \rightarrow Z$ are such that
$f^2$ and $f^2 \circ f^1$ are dihomotopy equivalences with surjective
homotopy inverses, 
then $f^2$ is a dihomotopy equivalence with surjective homotopy inverse.

\vskip .2cm

\vskip .2cm

\begin{example}
The unit segment is dicontractible. The wedge of two segments is dicontractible (which makes shows the version of dicontractibility discussed here to be notably different from that used in the framework of \cite{CSL16}). 
Note that in view of applications to directed topological complexity, this is coherent with
the fact that directed topological complexity should be invariant under directed homotopy equivalence. 
For any two pair of points in $\Gamma$ of directed segment or of a wedge of two segments, there is indeed
a continuous map depending on this pair of points to the unique dipath going from one to the other. 
\end{example}

\begin{example}
The Swiss flag is not directed homotopy equivalent to the hollow square.
This can be seen already using natural homology, that distinguishs the two, see
e.g. \cite{naturalhomology}. 

More precisely, we consider
the following d-spaces (SF on the left, HS on the right), coming from PV processes, which 
are subspaces of $\mathbb{R}^2$ and whose points are within 
the white part in the square (the grey part represents the forbidden states of the program) and whose dipaths are non decreasing paths for the componentwise ordering on $\mathbb{R}^2$. They are homotopy equivalent using the two maps (and even dmaps), $f$ from
$SF$ to $HS$ and $g$ from $HS$ to $SF$,  
depicted below ($f$ is the map on the left)~: 

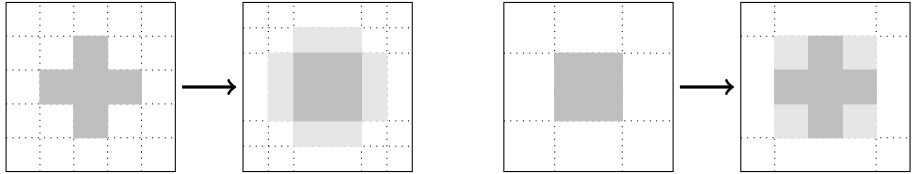
\begin{center}
\begin{tikzpicture}[auto,scale = 0.45]
\draw (0,0) rectangle (5,5);
\draw [dotted] (1,0) -- (1,5);
\draw[dotted] (2,0) -- (2,5);
\draw[dotted] (3,0) -- (3,5);
\draw[dotted] (4,0) -- (4,5);
\draw [dotted] (0,1) -- (5,1);
\draw[dotted] (0,2) -- (5,2);
\draw[dotted] (0,3) -- (5,3);
\draw[dotted] (0,4) -- (5,4);
\draw [fill = gray!50,draw = gray!50] (1,2) rectangle (4,3);
\draw [fill = gray!50,draw = gray!50] (2,1) rectangle (3,4);
\draw[->,very thick] (5.2,2.5) -- (6.8,2.5);
\draw (7,0) rectangle (12,5);
\draw [dotted] (7.75,0) -- (7.75,5);
\draw[dotted] (8.5,0) -- (8.5,5);
\draw[dotted] (10.5,0) -- (10.5,5);
\draw[dotted] (11.25,0) -- (11.25,5);
\draw [dotted] (7,0.75) -- (12,0.75);
\draw[dotted] (7,1.5) -- (12,1.5);
\draw[dotted] (7,3.5) -- (12,3.5);
\draw[dotted] (7,4.25) -- (12,4.25);
\draw[fill = gray!20,draw=gray!20] (7.75,1.5) rectangle (11.25,3.5);
\draw[fill = gray!20,draw=gray!20] (8.5,0.75) rectangle (10.5,4.25);
\draw [fill = gray!50,draw = gray!50] (8.5,1.5) rectangle (10.5,3.5);
\end{tikzpicture}
\quad\quad\quad
\begin{tikzpicture}[auto,scale = 0.45]
\draw (0,0) rectangle (5,5);
\draw [dotted] (3.5,0) -- (3.5,5);
\draw[dotted] (1.5,0) -- (1.5,5);
\draw [dotted] (0,1.5) -- (5,1.5);
\draw[dotted] (0,3.5) -- (5,3.5);
\draw [fill = gray!50,draw = gray!50] (1.5,1.5) rectangle (3.5,3.5);
\draw[->,very thick] (5.2,2.5) -- (6.8,2.5);
\draw (7,0) rectangle (12,5);
\draw [dotted] (11,0) -- (11,5);
\draw[dotted] (8,0) -- (8,5);
\draw [dotted] (7,1) -- (12,1);
\draw[dotted] (7,4) -- (12,4);
\draw [fill = gray!20,draw = gray!20] (8,1) rectangle (11,4);
\draw[fill = gray!50,draw=gray!50] (8,2) rectangle (11,3);
\draw[fill = gray!50,draw=gray!50] (9,1) rectangle (10,4);
\end{tikzpicture}
\end{center}

The points in light grey are the points which do not belong to the image of those maps. The problem is that those two programs are quite different: $SF$ has a dead-lock in $\alpha$ and inaccessible states, 
while $HC$ does not. Topologically, they do not have the same (directed) components in the sense of \cite{APCSII}.

It is easy to see that although $Pf$ and $Pg$ are homotopy equivalences as well, 
$(f,g)$ does not induce a dihomotopy equivalence in our sense. Take a point
$\alpha$ in the lower convexity of the Swiss flag, and consider the constant dipath
on $f(\alpha)$ in $HS$. $Pf$ maps the constant path on $\alpha$ onto the constant
path on $f(\alpha)$ but we can extend in $P(HS)^f$ this constant path to paths
$v$ from $f(\alpha)$ to the image by $f$ of the upper right point, which is again
the upper right point in $HS$. This extension makes the corresponding path
space within $HS$ homotopy equivalent to two points, whereas there are no path
from $\alpha$ to the right upper point in $SF$. 

In natural homology, $SF$ and $HS$ do not have bisimilar natural homologies since, 
considering the pair of points $(\alpha,\alpha)$ in $SF \times SF$, all extensions
of this pair of points will give 0th homology of there corresponding path space 
equal to $\Z$, whereas there
are extensions of any pair of points $(\beta,\beta)$ in $HS$ which give 0th homology
of there corresponding path space equivalent to $\Z^2$. 
\end{example}

\begin{example}
\begin{figure}[h]
\centering
	\quad
\begin{tikzpicture}[auto,scale = 0.65]
		\node (0) at (-0.05,-0.3) {$\footnotesize{\textbf{0}}$};
		\node (1) at (0.75,3.9) {$\footnotesize{\textbf{1}}$};
		\node (a) at (0.65,2.3) {$\footnotesize{\alpha}$};
		\coordinate (A) at (0,0);
		\coordinate (B) at (2,0.9);
		\coordinate (C) at (-1.3,1.8);
		\coordinate (D) at (0.7,2.6);
		\coordinate (A') at (0,1);
		\coordinate (B') at (2,1.9);
		\coordinate (C') at (-1.3,2.8);
		\coordinate (D') at (0.7,3.6);
		\draw (B) -- (B');
		\draw (A') -- (B');
		\draw (C) -- (C');
		\draw (A') -- (C');
		\draw (B') -- (D');
		\draw (C') -- (D');
		\draw [fill = gray,draw = black,opacity = 0.3] (A) -- (B) -- (B') -- (A') -- (A);
		\draw [fill = gray,draw = black,opacity = 0.3] (A) -- (C) -- (C') -- (A') -- (A);
		\draw [fill = gray,draw = black,opacity = 0.3] (A') -- (B') -- (D') -- (C') -- (A');
		\draw(A) -- (B);
		\draw (A) -- (C);
		\draw (B) -- (D);
		\draw (C) -- (D);
		\draw (D) -- (D');
		\draw (A) -- (A');
		\draw[->, very thick] (0.7,2.6) -- (0.7,3.6);
		\node (gamma) at (0.9,2.95) {$\footnotesize{\gamma}$};
		\end{tikzpicture}
		\quad\quad
\begin{tikzpicture}[auto,scale = 0.65]
		\node (0) at (-0.5,-0.3) {$\textcolor{white}{\footnotesize{\textbf{0}}}$};
		\node (1) at (0.75,3.9) {$\textcolor{white}{\footnotesize{\textbf{1}}}$};
		\coordinate (A'') at (0,0);
		\coordinate (B'') at (2,0.9);
		\coordinate (C'') at (-1.3,1.8);
		\coordinate (D'') at (0.7,2.6);
		\coordinate (A) at (0,0.5);
		\coordinate (B) at (2,1.4);
		\coordinate (C) at (-1.3,2.3);
		\coordinate (D) at (0.7,3.1);
		\coordinate (A') at (0,1);
		\coordinate (B') at (2,1.9);
		\coordinate (C') at (-1.3,2.8);
		\coordinate (D') at (0.7,3.6);
		\draw (B) -- (B');
		\draw (A') -- (B');
		\draw (C) -- (C');
		\draw (A') -- (C');
		\draw (B') -- (D');
		\draw (C') -- (D');
		\draw [fill = gray,draw = black,opacity = 0.3] (A) -- (B) -- (B') -- (A') -- (A);
		\draw [fill = gray,draw = black,opacity = 0.3] (A) -- (C) -- (C') -- (A') -- (A);
		\draw [fill = gray,draw = black,opacity = 0.3] (A') -- (B') -- (D') -- (C') -- (A');
		\draw(A) -- (B);
		\draw (A) -- (C);
		\draw (B) -- (D);
		\draw (C) -- (D);
		\draw (D) -- (D');
		\draw (A) -- (A');
		\draw[<-, thick] (B) -- (B'');
		\draw[dotted] (A'') -- (B'');
		\draw[<-, thick] (C) -- (C'');
		\draw[dotted] (A'') -- (C'');
		\draw[dotted] (B'') -- (D'');
		\draw[dotted] (C'') -- (D'');
		\draw [fill = gray,draw = black,opacity = 0.1] (A'') -- (B'') -- (B) -- (A) -- (A'');
		\draw [fill = gray,draw = black,opacity = 0.1] (A'') -- (C'') -- (C) -- (A) -- (A'');
		\draw[<-, thick] (D) -- (D'');
		\draw[<-, thick] (A) -- (A'');
		\draw[->,thick] (-0.65,0.9) -- (-0.65,1.4);
		\draw[->,thick] (1,0.45) -- (1,0.95);
		\draw[->,thick] (-0.3,2.2) -- (-0.3,2.7);
		\draw[->,thick] (1.35,1.75) -- (1.35,2.25);
		\end{tikzpicture}
		\quad\quad
\begin{tikzpicture}[auto,scale = 0.65]
		\node (0) at (0,-0.3) {$\textcolor{white}{\footnotesize{\textbf{0}}}$};
		\node (1) at (0.75,3.9) {$\textcolor{white}{\footnotesize{\textbf{1}}}$};
		\coordinate (A) at (0,0);
		\coordinate (B) at (2,0.9);
		\coordinate (C) at (-1.3,1.8);
		\coordinate (D) at (0.7,2.6);
		\coordinate (A') at (0,1);
		\coordinate (B') at (2,1.9);
		\coordinate (C') at (-1.3,2.8);
		\coordinate (D') at (0.7,3.6);
		\draw[->,thick] (B) -- (B');
		\draw (A') -- (B');
		\draw[->,thick] (C) -- (C');
		\draw (A') -- (C');
		\draw (B') -- (D');
		\draw (C') -- (D');
		\draw [fill = gray,draw = black,opacity = 0.1] (A) -- (B) -- (B') -- (A') -- (A);
		\draw [fill = gray,draw = black,opacity = 0.1] (A) -- (C) -- (C') -- (A') -- (A);
		\draw [fill = gray,draw = black,opacity = 0.3] (A') -- (B') -- (D') -- (C') -- (A');
		\draw[dotted](A) -- (B);
		\draw[dotted] (A) -- (C);
		\draw[dotted] (B) -- (D);
		\draw[dotted] (C) -- (D);
		\draw[->,thick] (D) -- (D');
		\draw[->,thick] (A) -- (A');
		\draw[->,thick] (-0.65,0.9) -- (-0.65,1.9);
		\draw[->,thick] (1,0.45) -- (1,1.45);
		\draw[->,thick] (-0.3,2.2) -- (-0.3,3.2);
		\draw[->,thick] (1.35,1.75) -- (1.35,2.75);
		\end{tikzpicture}
\caption{Naive equivalence between the Fahrenberg's matchbox $M$ and its upper face $T$}
\label{matchdef}
\end{figure}
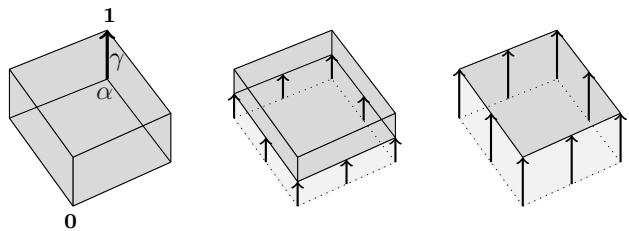
Consider Figure \ref{matchdef} : 
this depicts a naive dihomotopy equivalence between $M$ and to its upper face (so to a point). 
More precisely, the dmap $f$, which maps any point of $M$ to the point of $T$ just above of it, is a naive dihomotopy equivalence, whose inverse modulo dihomotopy is the embedding $g$ of $T$ into $M$. Hence, $f\circ g = id_T$ and a dihomotopy from $id_M$ to $g\circ f$ is depicted in Figure \ref{matchdef}.


Note that this naive dihomotopy equivalence, $(f,g)$, 
does not induce a dihomotopy equivalence in our sense. As a matter of fact, 
consider points $0$ and $\alpha$~: $\pth{X}(0,\alpha)$ is homotopy equivalent
to two points whereas $\pth{X}(f(0),f(\alpha))$ is homotopy equivalent to a point. 
\end{example}

\section{Dicontractibility and the dipath space map}

\begin{definition}
\label{def:weakdicontractible}
Let $X$ be a d-space. $X$ is said to be weakly dicontractible if its natural homotopy
(equivalently, natural homology \cite{naturalhomology}, by directed Hurewicz, 
all up to bisimulation) 
is the natural system - which we will denote by $\one$
- on 1 (the
final object in $Cat$) with value $\Z$ on the object, in dimension 0 and 0 
in higher dimensions. 
\end{definition}


The following is a direct consequence of Lemma \ref{lemma:bisim} but we give below a simple and
direct proof of it~: 

\begin{lemma}
Let $X$ be a dicontractible d-space.
Then
$X$ is weakly dicontractible in the sense of Definition \ref{def:weakdicontractible}. 
\end{lemma}

\begin{proof}
Suppose $X$ is dicontractible.
Therefore, by Definition \ref{def:dicontractible}, we have a continuous map $R \ : \{*\} \rightarrow PX$, 
continuously bigraded, 
which are are homotopy equivalences. 
\comment{
These homotopy equivalences are natural in the following sense. 
For each $(\gamma,\delta)$ morphism from $(c,d)$ to $(c',d')$ in $\pth{X}$, 
the following diagrams commute up to (classical) homotopy~: 

\begin{center}
  \begin{tikzpicture}[scale=1]
    \node (Fx') at (1,0) {$\scriptstyle \{*\}$};
    \node (Fx) at (1,1.5) {$\scriptstyle \{*\}$};
    \node (Gy') at (3,0) {$\scriptstyle \pth{X}(c',d')$};
    \node (Gy) at (3,1.5) {$\scriptstyle \pth{X}(c,d)$};
\node (H) at (2,.75) {$\scriptscriptstyle \sim$};
    \draw[->] (Fx) -- (Fx');
    \draw[->] (Fx) -- node [above] {$\scriptscriptstyle R_{c,d}$} (Gy);
    \draw[->] (Fx') -- node [below] {$\scriptscriptstyle R_{c',d'}'$} (Gy');
    \draw[->] (Gy) -- (Gy');
    \node (Fi) at (0.7,0.75) {$\scriptscriptstyle Id$};
    \node (Gj) at (3.3,0.75) {$\scriptscriptstyle PX(\gamma,\delta)$};
  \end{tikzpicture}
\end{center}

This shows that all the $\pth{X}(a,b)$ are contractible and 
}
All extension maps 
$(\gamma,\delta)$ induce identities modulo homotopy, trivially. 
Now, these diagrams of spaces induce, in homology, a diagram which has $\Z$
as value on objects for dimension 0, and 0 for higher dimensions. It is a simple
exercise to see that such diagrams are bisimilar to the one point diagram
which has only $\Z$ as value in dimension 0, and 0 in higher dimension.
This shows weak dicontractibility. 
\end{proof}





\begin{theorem}
Suppose $X$ is a contractible d-space. 
Then, the 
dipath space map has a continuous section if and only if $X$ is dicontractible. 
\end{theorem}

\begin{proof}
As $X$ is contractible, we have $f: X \rightarrow \{a_0\}$ (the constant map) and $g: \{a_0\} \rightarrow X$ (the inclusion)
which form a (classical) homotopy equivalence. Trivially, $f$ and $g$ are dmaps.

Suppose that we have a continuous section $s$ of $\chi$. 
There is an obvious inclusion map $i~: \{s(a,b)\} \rightarrow \pth{X}(a,b)$, which is 
continuously bigraded in $a$ and $b$. Define $R$ to be that map. 
Now the constant map $r~: \pth{X}(a,b) \rightarrow \{s(a,b)\}$ is a retraction map for $i$. 
We define 
$$ \begin{array}{lccc}
H \ : & \pth{X}\times [0,1] & \rightarrow & \pth{X} \\
& (u,t) & \rightarrow & \mbox{$v$ s.t. $\left\{\begin{array}{rcll}
v(x) & = & u(x) & \mbox{if $0\leq x \leq \frac{t}{2}$} \\
v(x) & = & s\left(u\left(\frac{t}{2}\right),u\left(1-\frac{t}{2}\right)\right)\left(\frac{x-\frac{t}{2}}{1-t}\right) & \mbox{if $\frac{t}{2} \leq x \leq 1-\frac{t}{2}$}\\
v(x) & = & u(x) & \mbox{if $1-\frac{t}{2}\leq x \leq 1$} \\
\end{array}\right.$}\\
\end{array}
$$
\noindent ($H(u,t)$ is extended by continuity for $t=1$ as being equal to $u$)
\comment{
$$ \begin{array}{lccc}
H_{a,b} \ : & \pth{X}(a,b)\times [0,1] & \rightarrow & \pth{X}(a,b) \\
& (u,t) & \rightarrow & u_{|\geq t} * s(a,u(t)) \\ 
\end{array}
$$
}

As concatenation and evaluation are continuous and as $s$ is continuous in both arguments
$H$ is continuous in 
$u \in PX$ and in $t$. $H$ induces families $H_{a,b} \ :  \pth{X}(a,b) \times [0,1] \rightarrow \pth{X}(a,b)$, and because $H$ is continuous in $u$ in the compact-open
topology, this family $H_{a,b}$ is continuous in $a$ and $b$ in $X$. 
Finally, we note that 
$H(u,1)=u$ and $H(u,0)=s(u(0),u(1))=i\circ r(u)$. Hence $r$ is a deformation retraction and
$PX(a,b)$ is homotopy equivalent to $\{s(a,b)\}$ and has the homotopy type we expect
(is contractible for all $a$ and $b$), meaning that $R$ is a continously bigraded homotopy equivalence.

The homotopy $H_{a,b}$ shows also that any extension map from $\pth{X}(a',b')$ to 
$\pth{X}(a,b)$ is homotopic
to the identity\footnote{Note the link with
the fact that $s(a,a')$ and $s(b',b)$ are Yoneda invertible.}. 
Therefore, $X$ is dicontractible.   

Conversely, suppose $X$ is dicontractible. 
We have in particular a continuous family (in $a$, $b$ in $X$)
of maps $R_{a,b} \ : \ \{ * \} \rightarrow \pth{X}(a,b)$. Define 
$s(a,b)=R_{a,b}(*)$, this is a continuous section of $\chi$. 
\end{proof}

\paragraph{Remark : }
Sometimes, we do not know right away, in the theorem above, that $X$ is contractible. But instead, 
there is  an initial state in $X$, i.e. a state $a_0$
from which 
every point of $X$ is reachable. 
Suppose then that, as in the Theorem above, $\chi$ has a continuous section $s : \Gamma_X \rightarrow PX$. 
Consider $s'(a,b)=s^{-1}(a_0,a)*s(a_0,b)$ the concatenation of the inverse dipath,
going from $a$ to $a_0$, with the dipath going from $a_0$ to $b$~: this is a continuous
path from $a$ to $b$ for all $a$, $b$ in $X$. Now, $s'$ is obviously continuous since
concatenation, and $s$, are. By a classical theorem \cite{farber}, this implies that
$X$ is contractible and the rest of the theorem holds. 

\vskip .2cm

\section{Directed topological complexity}

\begin{definition}
The directed topological complexity $\diTC{X}$ 
of a d-space $X$ (which is also an Euclidean Neighborhood retracts, or ENR)
is the minimum number $n$ (or $\infty$
if no such $n$ exists) such
that $\Gamma_X$ can be partitioned into $n$ ENRs
\footnote{This could be
replaced as in the classical case, by asking for a covering by open sets, or by 
a covering by closed sets.}
$F_1,\ldots,F_n$
such that there exists a map $s : \Gamma_X \rightarrow PX$ (not necessarily continuous, of
course!) with~: 
\begin{itemize}
\item $\chi \circ s = Id$ ($s$ is a, non-necessarily continuous, section of $\chi$)
\item $s_{|F_i} : F_i \rightarrow PX$ is continuous
\end{itemize}
\end{definition}


\begin{example}
Consider the PV program $Pa Va \mid Pa Va$ ($a$ is a mutex), the component
category (see e.g. \cite{APCSII}) has 4 regions $C_1$, $C_2$, $C_3$ and $C_4$, with a unique morphism from 
$C_1$ to $C_2$, $C_1$ to $C_3$, $C_2$ to $C_4$ and $C_3$ to $C_4$ (but two from $C_1$
to $C_4$). 
We have $\diTC{X}=2$ since $\Gamma_X$ can be partitioned into 
$\{(x,y) \ | \ x \in C_i, y \in C_j, (i,j)\neq (1,4)\}$ and $C_1\times C_4$.
More generally speaking, the lifting property of components, Proposition
7 of \cite{APCSI}, in the case when
we have spaces $X$ with components categories (such as with the cubical complexes
of \cite{CSL16}) implies that we can examine the dihomotopy type of $X$ 
through the dihomotopy equivalent space, quotient of $X$ by its components. 
\end{example}

\begin{example}
Consider now $X$ to be 
a cube minus an inner cube, seen as a partially ordered space with
componentwise ordering, and as such, as a d-space. 
Note that it is dihomotopy equivalent to a hollow cube, hence we will examine
its directed topological complexity through this diretract. Up to homotopy, this is $S^2$ (seen as a sphere with unit radius centered in 0,
in $\R^3$), for which we know (e.g. \cite{farber}) that
$TC(X)=3$. A simple partition which shows that it is at most 3 is
\begin{itemize}
\item $F_1=\{(x,y) \ | \ x \neq -y\}$
\item $F_2=\{(x,-x) \ | \ x \neq x_0\}$
\item $F_3=\{(x_0,-x_0) \}$
\end{itemize}
where $x_0$ is some point of $S^2$ that we can choose, for which we have a 
smooth vector field $v$ on $S^2$, which is non zero everywhere except at $x_0$ (such 
a point must exist by general theorems, and we can find a vector field which will
only be zero at one point). 
On $F_1$, we take as section to the classical path space fibration, the Euclidean
geodesic path from $x$ to $y$. On $F_2$, we take the path from $x$ to $-x$
which follows the vector field $v$. Finally, we take any path from $x_0$ to $-x_0$. 

This partition can be used to find a upper bound for the directed topological 
complexity of the hollow cube (or equivalently, the cube minus an inner cube). 
Note that in general, the two notions are incomparable. Because if
$X\times X=\cup_{i=1}^k F_i$ ($k$ is $TC(X)$) with $s_{|F_i}$ continuous section of the path space
fibration, then we do not know if on each $F_i$ we can find a continuous
section of the dipath space map. Conversely, if $\Gamma_X=\cup_{i=1}^l F_i$
($l$ is $\diTC{X}$), each $s_{|F_i}$ provides us with a continuous section, locally
to $F_i$, of the dipath space map and of the path space fibration, but unfortunately,
we can only cover that way $\Gamma_X \subseteq X \times X$. These two notions
are only trivially comparable when the set of directed paths is equal to the 
set of all continuous paths, in which case the directed topological complexity
is equal to the classical topological complexity.

Still, in some cases, such as in the case of the hollow cube, we can carefully
examine the partition given by the classical topological complexity, to, in general,
find an upper bound to the directed topological complexity. 
Here, we can choose $x_0$ to be the final point of $X$ (the hollow cube), and
strip down $F_1$, $F_2$ and $F_3$ to be the part of the previous partition, intersected
with $\Gamma_X$. On each of these three sets we have a trivial section to the
dipath space map. Hence $\diTC{X}\leq 3$. 

Martin Raussen observed\footnote{Private communication during the Hausdorff
Institute ``Applied Computational Algebraic Topology'' semester, on the 14th
September 2017.} that in fact, $\diTC{X}=2$ which shows an essential difference
to the classical case. 
\end{example}


\paragraph{Remark : } 
Consider the universal covering of an NPC cubical complex (as in \cite{ardila,cat0}). 
It is CAT(0). We conjecture that $\diTC{X}$ is the number of maximal configurations of
the corresponding prime event structure (see \cite{chepoi1,chepoi2,winskel,ardila}, 
or equivalently, the number of levels in 
the universal dicovering (see e.g. \cite{dicoverings}), or the number of maximal dipaths starting in the initial
point, up to dihomotopy). 

\begin{lemma}
Let $X$ and $Y$ be two dihomotopy equivalent spaces. Then $\diTC{X}=\diTC{Y}$.
\end{lemma}

\begin{proof}
As $X$ and $Y$ are dihomotopy equivalent, we have $f \ : \ X \rightarrow Y$ and 
$g \ : \ Y \rightarrow X$ dmaps, which form a homotopy equivalence between $X$ and $Y$. 
We also get $G_{c,d} : PX(g(c),g(d)) \rightarrow PY(c,d)$ which is inverse modulo homotopy
to $Pg_{c,d}$, and varies continuously according to $c$, $d$ ; and $F_{a,b} \ : \ PX(a,b)
\rightarrow PY(f(a),f(b))$ which is inverse modulo homotopy to $Pf_{a,b}$, varying continuously
according to $a$ and $b$. 
We note that $f$ and $g$ induce continuous maps $f^* \ : \ \Gamma_X \rightarrow \Gamma_Y$
and $g^* \ : \ \Gamma_Y \rightarrow \Gamma_X$ ($\Gamma_X$ and $\Gamma_Y$ inherit the product
topology of $X$, resp. $Y$). 

Suppose first $k=\diTC{X}$. Thus we can write $\Gamma_X=F^X_1 \cup \ldots \cup F^X_k$ such 
that we have a map $s \ : \ \Gamma_X \rightarrow PX$ with $\chi \circ s=Id$ and $s_{|F^X_i}$
is continuous. 

Define $F^Y_i=\{ u \in \Gamma_Y \ | \ g^*(u)\in F^X_i\}$ (which is an ENR - or open 
if we choose the alternate definition - as $F^X_i$
is ENR and $g^*$ is continuous) and define $t_{|F^Y_i}(u)=G_{u} \circ s_{|F^X_i}
\circ g^*(u)\in PY(u)$ for all $u \in F^Y_i \subseteq \Gamma_Y$. This is a continuous map in $u$ since
$s_{|F^X_i}$ is continuous, $g^*$ is continuous, and $G$ is continuously bigraded. Therefore
$\diTC{Y} \leq \diTC{X}$. 

Conversely, suppose $l:\diTC{Y}$, $\Gamma_Y=F^Y_1 \cup \ldots \cup F^Y_l$
such 
that we have a map $t \ : \ \Gamma_Y \rightarrow PY$ with $\chi \circ t=Id$ and $t_{|F^Y_i}$
is continuous. Now define
$F^X_i=\{ u \in \Gamma_X \ | \ f^*(u)\in F^Y_i\}$ (which is an ENR - or an open 
set if we choose the alternate definition - as $F^Y_i$ is ENR and $f^*$ is continous) 
and define $s_{|F^X_i}(u)=F_{u} \circ t_{|F^Y_i}
\circ f^*(u)\in PX(u)$ for all $u \in F^X_i \subseteq \Gamma_X$. This is a continuous map in $u$ since
$t_{|F^Y_i}$ is continuous, $f^*$ is continuous, and $F$ is continuously bigraded. Therefore
$\diTC{X} \leq \diTC{Y}$. Hence we conclude that $\diTC{X}=\diTC{Y}$ and directed topological
complexity is an invariant of dihomotopy equivalence.  
\end{proof}

\paragraph{Remark : }
The proof above is enlightening in that it uses all homotopy equivalences generated
by the data that $f$ is a dihomotopy equivalence. We would have had only a
Dwyer-Kan type of equivalence, we would not have had that directed topological
equivalence is a dihomotopy invariant.

\vskip .2cm

\section{Conclusion}

There are numerous developments to this, in studying directed topological complexity
with a view to control theory, but also on the more fundamental level of the
structure of dihomotopy equivalences. 

There is for instance an interesting notion of weak-equivalence, coming out of
our dihomotopy equivalence (same conditions, but inducing isomorphisms of the
fundamental groups of the different path spaces, with similar extension conditions). 
This weak-equivalence should have good properties with respect to our dihomotopy
equivalence, have the 2-out-of-6 property, and a refined form of natural homology
should come out as a derived functor in that framework. This is left for another
venue. 

\paragraph{Acknowledgments}
We thank the participants of the Hausdorff Institute Seminar Institute ``Applied Computational Algebraic Topology'' for useful comments on this work, and most particularly
Martin Raussen, Lisbeth Fajstrup and Samuel Mimram. We thank Tim Porter for his
careful reading of a previous version of this report, and his comments and ideas. 



\end{document}